\documentclass[12pt]{article}%,a4]{article}
\usepackage{setspace}
\usepackage{amsmath,amssymb}
\usepackage[bf,small,tableposition=top]{caption}
\usepackage{subfig}
\usepackage{amsthm}
\usepackage{lineno}
\usepackage[figuresright]{rotating}
\usepackage{enumerate}
\usepackage{authblk}
\usepackage[T1]{fontenc}
\usepackage[utf8]{inputenc}
\usepackage{graphics}
\usepackage{graphicx}
\usepackage{multicol}
\usepackage{lscape}
\usepackage[dvipsnames]{xcolor}
\usepackage{mathrsfs}
\usepackage{epstopdf}
\usepackage{hyperref}
\usepackage{textcomp}
\usepackage{listings}
\usepackage{float}

\usepackage{tipa}
\usepackage{verbatim}
\usepackage{diagbox}
\usepackage{multirow,bigdelim,dcolumn,booktabs}
\newtheorem{prop}{Proposition}
\newtheorem{thm}{Theorem}%[section]
\theoremstyle{definition}

\theoremstyle{remark}
 \newtheorem{cor}{Corollary}

 \usepackage{multirow}

\usepackage{array} % Include the array package

\newcolumntype{H}{>{\setbox0=\hbox\bgroup}c<{\egroup}@{}}
\usepackage{booktabs}
\usepackage[authoryear]{natbib}
\title{Laplace Transform driven Stein-type Goodness-of-fit Tests for Pareto Distribution}

\author
{Deepesh Bhati$^{1}$, Sakshi Khandelwal$^{1}$\footnote{Corresponding author: gsakshi1506@gmail.com} \\
\normalsize{$^{1}$Department of Statistics, Central University of Rajasthan, Kishangarh, India} \\
}

\date{}

\begin{document} 


\baselineskip13pt

\maketitle 

% Place your aract within the special {sciabstract} environment.

\begin{abstract}
The Pareto distribution plays a crucial role in various disciplines, necessitating robust goodness-of-fit tests for its validation. This article introduces a novel tests based on Stein’s characterization and the Laplace transform, offering a fresh perspective on model assessment. We establish the asymptotic properties of the proposed test and evaluate its empirical performance against existing methods in terms of size and power. Our findings demonstrate that the new test often outperforms or performs comparably to established tests. In addition, real data applications illustrate its practical utility.
\end{abstract}



\section{Introduction} 
The Pareto distribution, named after the Italian economist \citet{pareto1896cours}, is a prominent member of the class of heavy-tailed distributions. Initially developed to describe the distribution of wealth in society, where a small fraction of the population controls the majority of resources, it has since found applications in diverse areas including economics (\citet{simon1955class}), finance (\citet{rydberg2000realistic}), insurance (\citet{embrechts1997modelling}), and environmental science (\citet{katz2002statistics}). Its characteristic heavy tail makes it particularly suitable for modeling extreme events, rare occurrences, and other phenomena where large deviations play a crucial role.\\
\indent Given its central role in modeling and risk assessment, verifying the suitability of the Pareto distribution for a dataset is crucial. This has led to the development of numerous goodness-of-fit (gof) tests, which assess whether empirical data are consistent with the Pareto law. These include classical methods based on the empirical distribution function (\citet{choulakian2001goodness}), probability plots (\citet{castillo2012extreme}), and empirical characteristic functions (\citet{ndwandwe2023new}), to name just a few. A particularly fruitful approach has been the use of characterization-based tests, which exploit properties uniquely satisfied by the Pareto distribution (see \citet{malik1970characterization, ahsanullah1974characterization, nofal2017new, Ngatchou}). These methods offer strong discriminatory power by focusing on features intrinsic to the distribution.\\
\indent One such modern framework is Stein’s method, a powerful tool originally introduced for normal approximation (\citet{stein1972bound}). It characterizes a distribution via a differential equation that is uniquely satisfied by the target distribution. This method has since been extended to non-normal cases, including distributions on the entire real line (e.g., Gaussian, Laplace), semi-bounded distributions (e.g., Pareto, Gamma, Exponential, Weibull), and bounded distributions (e.g., Uniform, Beta) (see \citet{betsch2021fixed}). Notably, these extensions have led to developments for Lévy (\citet{kumari2023goodness}), Weibull (\citet{ebner2023goodness}), and Pareto families (\citet{bhati2025new}), enabling the construction of robust, tailored goodness-of-fit tests. In the context of the Pareto distribution, Stein-type characterizations leverage its density structure to yield insightful test statistics.\\
\indent Complementary to this, the Laplace transform has long been used in statistical theory for its ability to simplify complex problems. It offers compact representations of distributions, facilitates moment-based analysis, and enables the derivation of identities that are useful in estimation and testing (\citet{widder1941laplace}). In recent years, the empirical Laplace transform has gained attraction in gof testing, proving useful in constructing asymptotically powerful tests for distributions with heavy tails or skewness (see \citet{baringhaus1991class, henze1992new, jimenez2009goodness}).\\
\noindent Motivated by these developments, we propose a novel gof test for the Pareto distribution that integrates Stein’s characterization with the Laplace transform. This hybrid approach combines the differential structure inherent in Stein’s method with the integral-based flexibility of the Laplace transform. To the best of our knowledge, this is the first such test for the Pareto distribution, filling a gap in the existing literature and offering potential benefits in applications spanning economics, finance, and reliability engineering.\\
\indent Despite the wide array of available tests, the need for new procedures remains strong. Different characterizations capture different aspects of a distribution’s behavior, and a single counter-example may be enough to challenge a distributional assumption (\citet{nikitin1995asymptotic, Ngatchou}). Our proposed method contributes to this landscape by offering a new perspective that blends analytical rigor with practical effectiveness.\\
\indent In this work, we consider a composite null hypothesis where the shape parameter of the Pareto distribution is unknown. Let $\mathcal{P} = \{ P(\alpha, \beta) \mid \alpha > 0, \beta > 0 \}$ denote the family of Pareto distributions with shape $\alpha$ and scale $\beta$, defined by the cumulative distribution function and density function:

\begin{equation}
\label{Pab}
F(x;\alpha,\beta) = 
    \begin{cases}
      1-\left(\dfrac{x}{\beta}\right)^{-\alpha}, & x \geq \beta, \\
      0, & x < \beta,
    \end{cases}
\quad , \quad
f(x;\alpha,\beta) = 
    \begin{cases}
     \alpha \beta^{\alpha} x^{-\alpha-1}, & x \geq \beta,\\
      0, & \text{otherwise}.
    \end{cases}
\end{equation}
\noindent In many practical scenarios, the scale parameter $\beta$ is known or can be standardized to 1. This leads to the simplified form:
\begin{equation}
\label{Pa}
F(x;\alpha) = 
    \begin{cases}
      1-\dfrac{1}{x^\alpha}, & x \geq 1,\\
      0, & \text{otherwise}.
    \end{cases}
\quad , \quad
f(x;\alpha) = 
    \begin{cases}
     \frac{\alpha}{x^{\alpha+1}}, & x\geq1,\\
      0, & \text{otherwise}.
    \end{cases}
\end{equation}
\noindent We denote this standard form by $P(\alpha)$ and assume $\beta=1$ throughout, without loss of generality. Given a random sample $X_1, \ldots, X_n$ drawn independent and identically distributed (iid) from an unknown distribution $\mathbb{P}_X$, we aim to test the hypothesis:
\begin{equation}
\mathscr{H}_0: \mathbb{P}_X \in \mathcal{P},
\end{equation}
against general alternatives.\\
\indent In what follows, we develop a new class of gof tests grounded in a novel Stein-type characterization of the Pareto distribution, expressed via the Laplace transform. This method builds upon recent work in Weibull gof testing (\citet{ebner2023goodness}), but adapts the idea to suit the Pareto context. To evaluate the performance of the proposed test, we analyze its power under a variety of alternatives, following standard procedures in gof test evaluation.\\
\indent The rest of the paper is structured as follows: Section \ref{Fusion} presents the underlying characterization motivating our test. Section \ref{proposed} introduces the proposed test statistic, while Section \ref{limit} investigates its asymptotic properties. In Section \ref{MCstudy}, we conduct a Monte Carlo study comparing our test with existing ones. Section \ref{data} demonstrates practical applications to real data, and Section \ref{conclusions} concludes the paper.\\
\indent For clarity, we define notation used throughout: all limits are taken as $n \to \infty$; $\stackrel{d}{\longrightarrow}$ and $\stackrel{P}{\longrightarrow}$ denote convergence in distribution and in probability, respectively; $\mathbb{E}$ and $\mathbb{V} \text{ar}$ denote expectation and variance; $o_p(1)$ refers to a term that converges to zero in probability; and $\mathbb{I}(A)$ is the indicator function of set $A$.

%-----------------------------------
%\section{Amalgamation of Stein's Characterization and Laplace Transform} \label{Fusion}
\section{Amalgamation of Laplace Transform in Stein’s Characterization} \label{Fusion}
A common approach in constructing gof tests, which has become popular in recent years, is to utilize a characterization of the family of distributions being considered. In this article we have fused fixed-point characterization of Pareto distribution (see \citet{betsch2021fixed}) with the Laplace transform. 

Theorem 2 of \citet{betsch2021fixed} states that, if $f: \mathbb{R} \rightarrow [0,\infty)$ be a probability density function with $\text{spt}(f) = [L,\infty), \, L> -\infty $, then $X \sim f$ if, and only if, 
\begin{equation}
    f_X(s)= \mathbb{E} \left( -\frac{f'(x,\alpha)}{f(x,\alpha)} \mathbb{I}(X>s) \right), \, s>L.
\end{equation}

\noindent If $X$ is from Pareto distribution with scale parameter 1, then this fixed-point characterization reduces to 
\begin{equation} \label{FPC}
   f_X(s)= \mathbb{E} \left( \frac{\alpha+1}{X} \mathbb{I}(X>s) \right), \, s>1. 
\end{equation}

\noindent The Laplace transform for any distribution function is defined as:
\begin{align} \label{LT}
\mathscr{L}_X(t) &= \mathbb{E} (\exp(-Xt)) \nonumber \\
     &= \int_x \exp(-st) \cdot f_X(s) \, ds
\end{align}

\noindent To fuse Stein's characterization with Laplace transform we will replace the density function in the Equation \eqref{LT} to the Equation \eqref{FPC}. 
This fusion gives us 
\begin{align} \label{NC}
       \mathscr{L}_X(t) &= \int_1^\infty \exp(-st) \cdot \mathbb{E} \left( \frac{\alpha+1}{X} \mathbb{I}(X>s) \right) \, ds, \nonumber \\
       &= \mathbb{E} \bigg( \frac{\alpha+1}{X} \int_1^\infty \exp(-ts) \mathbb{I} (X>s) \, ds\bigg), \nonumber \\
       &= \mathbb{E} \bigg( \frac{\alpha+1}{X} \cdot  \frac{\exp(-t)-\exp(-tX)}{t} \bigg).
\end{align}

\noindent Hence, to develop a new gof test we will make use of the theorem which is given as follows. 
\begin{thm}
 Let $\alpha$>0 and $X$ be a positive random variable with differentiable density $f$ and Laplace transform  $\mathscr{L}_X$ satisfying $\mathbb{E} \bigg| \frac{ X \frac{d}{dx} f(x) |_X}{f(X)}  \bigg| < \infty$. Then $X$ has a $P(\alpha)$ distribution if and only if 
 \begin{equation}
     \mathscr{L}_X(t) = \mathbb{E} \bigg( \frac{\alpha+1}{X} \cdot  \frac{\exp(-t)-\exp(-tX)}{t} \bigg)
 \end{equation}
 for each $t>0$. 
\end{thm}
\begin{proof}
We first assume that $X$ has the Pareto distribution $P(\alpha)$, and we write $F(\cdot \, ; \alpha)$ and $f(\cdot \, ; \alpha)$ for the distribution function and the density function of $X$, respectively. \\

\noindent Furthermore, it is obvious that the density function $f$ is continuous and positive on $(1,\infty)$, and there exists a partition $L<y_1<\ldots<y_m<\infty$ such that $f$ is continuously differentiable on $(1,y_1),(y_k,y_{k+1}), k \in \{1,\ldots,m-1\}$ and $(y_m,R)$, where $S(p)=(1,\infty) \backslash \{y_1,\ldots,y_m \}$.

\noindent Moreover, we put 
$k_f(x) = \bigg| \frac{f'(x,\alpha) \, \min (F(x,\alpha),1-F(x,\alpha) )}{f^2(x,\alpha)}  \bigg|, \quad 1<x<\infty .$ 

\noindent Letting $\eta= 2^{1/\alpha}$, we have 
\begin{equation*} \min (F(x,\alpha),1-F(x,\alpha) ) = 
    \begin{cases}
        F(x,\alpha), \qquad ~~~ x \leq \eta \\ 
        1-F(x,\alpha) , \quad x>\eta 
    \end{cases} \, .
\end{equation*}

\noindent It is easy to show that $\lim _{x\rightarrow 1} k_f(x) = 0$ and $\lim _{x\rightarrow \infty} k_f(x) = \bigg| - \frac{(\alpha+1)}{\alpha} \bigg| $. The continuity of $k_f(\cdot)$ then yields
\begin{equation} \label{C2}
    \sup_{x \in (1,\infty) }k_f(x) < \infty.
\end{equation}
Furthermore, 
\begin{equation} \label{C3}
    \int_1^\infty (1+|x|) \, |f'(x,\alpha)| \, dx = 1+2\alpha < \infty
\end{equation}
and
\begin{equation} \label{C4}
    \lim_{x\rightarrow1} \frac{F(x,\alpha)}{f(x,\alpha)} = 0
\end{equation}

\noindent In view of Equation \eqref{C2}, \eqref{C3} and \eqref{C4}, we can apply Theorem 2 of \citet{betsch2021fixed}. Hence $X$ follows a $P(\alpha)$ distribution if and only if its density is given by 
\[ f_X(s)= \mathbb{E} \left( -\frac{f'(x,\alpha)}{f(x,\alpha)} \mathbb{I}(X>s) \right) = \mathbb{E} \left( \frac{\alpha+1}{X} \mathbb{I}(X>s) \right)  \]
for almost every $s>1$. Next, we apply Tonelli's theorem to conclude that
\begin{align*}
    \int_1^\infty \exp(-ts) \, \mathbb{E} \bigg( \bigg| - \frac{f'(x,\alpha)|_X}{f(X,\alpha)} \bigg| \, \mathbb{I} (X>s) \bigg) &= \int_1^\infty \exp(-ts) \, \int_1^\infty  \bigg| f'(x,\alpha)\bigg| \, \mathbb{I} (x>s) \, dx \, ds, \\
    &= \int_1^\infty \bigg( \frac{\exp(-t) - \exp(-tx)}{t} \bigg)  \bigg| f'(x,\alpha)\bigg| \, dx ,\\
    & \leq \int_1^\infty (x-1) \bigg| f'(x,\alpha)\bigg| \, dx, \\ 
    & \leq 1 
\end{align*}
holds for t>0. Since the integral of the absolute value of the function is finite, we can use the Fubini's theorem, and the Laplace transform of $X$ takes the form
\begin{align*}
 \mathscr{L}_X(t) &= \int_1^\infty \exp(-ts) \, \mathbb{E} \bigg( \frac{\alpha+1}{X} \mathbb{I} (X>s) \bigg) \, ds, \\ &=  \mathbb{E} \bigg( \frac{\alpha+1}{X} \cdot  \frac{\exp(-t)-\exp(-tX)}{t} \bigg)  
\end{align*}
for $t>0$. The converse statement also follows since the Laplace transform determines the distribution.
\end{proof}

%---------------------------------------------------------------

\section{Test statistic} \label{proposed}  
Let $X_1, X_2,\ldots, X_n$ be a sample from a continuous df $F$ defined on $[1,\infty)$. 
On the basis of above amalgamation of Stein's characterization with Laplace transform given as in Section \ref{Fusion}, we will construct three departure measures (Integral type,  Kolmogorov-Smirnov type and cramer-Von Mises type) $\Delta_1(F), \Delta_2(F) \text{ and } \Delta_3(F)$ to test the hypothesis, which are given by: 

\begin{equation}
    \Delta_1(F) = \int_0^\infty \bigg(  \mathbb{E} \bigg( \frac{ \hat \alpha+1}{X} \cdot  \frac{\exp(-t)-\exp(-tX)}{t} \bigg) - \mathscr{L}_X(t)  \bigg) \bigg) \, dt
\end{equation}

\begin{equation}
    \Delta_2(F) = \sup_{t>0} \bigg|  \mathbb{E} \bigg( \frac{\hat \alpha+1}{X} \cdot  \frac{\exp(-t)-\exp(-tX)}{t} \bigg) - \mathscr{L}_X(t)  \bigg| 
\end{equation}

\begin{equation}
    \Delta_3(F) =  \int_0^\infty \bigg(  \mathbb{E} \bigg( \frac{\hat \alpha+1}{X} \cdot  \frac{\exp(-t)-\exp(-tX)}{t} \bigg) - \mathscr{L}_X(t)  \bigg)^2 \, dt
\end{equation}
respectively, where $\hat{\alpha}$ is MLE of $\alpha$ which is given as 
\begin{equation}
    \label{alpha_hat}
    \hat{\alpha}_n :=\hat{\alpha}_n(X_1,\ldots,X_n)= \frac{n}{\sum_i^n \log(X_i)}.
\end{equation}
Note that $X$ has support on $[1,\infty)$ with df $\mathcal{F}$.
Therefore, based on the characterization given in the previous section, $\Delta_1(F)=0$, $\Delta_2(F)=0$ and $\Delta_3(F)=0$ if and only if $f$ is the Pareto distribution. However, these quantities are unknown, and a test should be constructed based on the empirical counterpart of them.

\noindent Hence, to develop the test statistics based on the above departure measures we will make use of the empirical distribution function of $ \mathbb{E} \bigg(\frac{\hat{\alpha}+1}{X} \cdot  \frac{\exp(-t)-\exp(-tX)}{t} \bigg)$ and $\mathscr{L}_X(t) $, which are given as: 
\[ \mathbb{E} \bigg(\frac{\hat{\alpha}+1}{X} \cdot  \frac{\exp(-t)-\exp(-tX)}{t} \bigg) = \frac{1}{n} \sum_{i=1}^n \frac{\hat{\alpha}+1}{X_i} \cdot \bigg( \frac{\exp(-t)-\exp(-tX_i)}{t} \bigg) \]
and
\[ \mathscr{L}_X(t)  = \frac{1}{n} \sum_{i=1}^n \exp(-tX_i) \]

\noindent Finally, we propose the following three classes of test statistics:
\begin{align} \label{TSI}
    \mathscr{DS}_1 &= \int_0^\infty \bigg (  \frac{1}{n} \sum_{i=1}^n \frac{\hat{\alpha}+1}{X_i} \cdot \bigg( \frac{\exp(-t)-\exp(-tX_i)}{t} \bigg) -\frac{1}{n} \sum_{i=1}^n \exp(-tX_i)   \bigg) \, dt , \nonumber \\
    &= \frac{\hat{\alpha}+1}{n} \sum_{i=1}^n \frac{\log(X_i)}{X_i} - \frac{1}{n} \sum_{i=1}^n \frac{1}{X_i}.
\end{align}

\begin{align*}
    \mathscr{DS}_2 &= \sup_{t>0} \bigg|  \frac{1}{n} \sum_{i=1}^n \frac{\hat{\alpha}+1}{X_i} \cdot \bigg( \frac{\exp(-t)-\exp(-tX_i)}{t} \bigg) -\frac{1}{n} \sum_{i=1}^n \exp(-tX_i) \bigg| 
\end{align*}

\noindent Note that using the transformation $\exp(-t) =s$ it is easily concluded that $\mathscr{DS}_2$ reduces to:
\begin{align}
    \mathscr{DS}_2 &= \sup_{s \in [0,1]} \bigg| \frac{1}{n} \sum_{i=1}^n \frac{\hat{\alpha}+1}{X_i} \bigg( \frac{s-s^{X_i}}{-\log(s)}\bigg) -\frac{1}{n} \sum_{i=1}^n s^{X_i}   \bigg|.
\end{align}
and,
\begin{align}
    \mathscr{DS}_3 =& \int_0^\infty \bigg (  \frac{1}{n} \sum_{i=1}^n \frac{\hat{\alpha}+1}{X_i} \cdot \bigg( \frac{\exp(-t)-\exp(-tX_i)}{t} \bigg) -\frac{1}{n} \sum_{i=1}^n \exp(-tX_i)   \bigg)^2 \, dt , \nonumber \\
    =& \frac{(\hat \alpha+1)^2}{n^2} \sum_{i,j=1}^n \frac{1}{X_iX_j} \bigg( X_i \log(X_i+X_j) + X_j \log(X_i+X_j)- X_i \log(1+X_i) \nonumber \\ &- X_j\log(1+X_j) - \log\bigg(\frac{(1+X_i)(1+X_j)}{4}\bigg) \bigg) - \frac{\hat \alpha+1}{n^2} \sum_{i,j=1}^n \frac{1}{X_i} \log\bigg( \frac{X_i+X_j}{1+X_j} \bigg) \nonumber \\& - \frac{\hat{\alpha}+1}{n^2} \sum_{i,j=1}^n \frac{1}{X_j} \log\bigg( \frac{X_i+X_j}{1+X_i} \bigg) + \frac{1}{n^2} \sum_{i,j=1}^n \frac{1}{X_i+X_j}\, .
\end{align}
\section{Asymptotic Theory} \label{limit} In this section we study the asymptotic properties of the proposed test under the null hypothesis. It is well known that under $\mathscr{H}_0$, the MLE of shape parameter is given as in Eq.~\eqref{alpha_hat}. \\
From this under $\mathscr{H}_0$, one can deduce the following equality:

\begin{align*}
    \sqrt{n} (\hat{\alpha}-\alpha) =   \frac{1}{\sqrt{n}} \sum_{i=1}^n I_\alpha^{-1} \,  \dot{l} (X_i \, ;\alpha) + o_P(1),
\end{align*}
where,
$I_\alpha^{-1}$ is the Fisher information and $\dot{l} (X_i \, ;\alpha)$ is the score function which are given as
\begin{align*}
I_\alpha &= - \frac{\partial^2}{\partial \alpha^2} \log f(x), \, \text{and}   \\ 
\dot{l}(X_i\, ; \alpha )&= \frac{\partial}{\partial \alpha} \log f(x),
\end{align*}
respectively. Hence we have the linear representation,
\begin{equation} \label{LRA}
 \sqrt{n} (\hat{\alpha}-\alpha) = \frac{-\alpha^2}{\sqrt{n}} \sum_{i=1}^n  \, \bigg( \log X_i - \frac{1}{\alpha} \bigg) + o_P(1)   
\end{equation}

\subsection{Asymptotic property of the test statistics \texorpdfstring{$\mathscr{DS}_1$}{DS1} under \texorpdfstring{$\mathscr{H}_0$ }{H0}}\label{limit1} 
\begin{thm} Let $X_1,\ldots,X_n$ be iid. $P(\alpha)$ for $\alpha>0$. Then we have
\begin{equation*}
    \sqrt{n}\mathscr{DS}_1\stackrel{\mathcal{D}}{\rightarrow} N(0,\sigma^2),
\end{equation*}
where $\sigma^2=\frac{\alpha [2+(2 \alpha + \alpha^2 ) (5+8\alpha+4\alpha^2)]}{(2+\alpha)^3 (1+\alpha)^4 }$.
\end{thm}
\begin{proof}
Write
\begin{align*}
\overline{Y}_n:&=\frac{1}{n}\sum_{i=1}^n \frac{\log X_i}{X_i},
\qquad
\mu_Y:=\mathbb E\bigg[\frac{\log X_1}{X_1}\bigg]=\frac{\alpha}{(1+\alpha)^2},    \\
\overline{Z}_n:&=\frac{1}{n}\sum_{i=1}^n \frac{1}{X_i},
\qquad
\mu_Z:=\mathbb E\bigg[\frac{1}{X_1}\bigg]=\frac{\alpha}{1+\alpha}.
\end{align*}
\noindent Since 
\[(1+\alpha) \frac{\alpha}{(1+\alpha)^2} - \frac{\alpha}{1+\alpha} = 0,\]
we obtain
\[
\mathscr{DS}_1=(\widehat{\alpha}-\alpha)\overline{Y}_n+(1+\alpha)(\overline{Y}_n-\mu_Y)-(\overline{Z}_n-\mu_Z).
\]
Hence
\begin{equation}\label{eq:DS1compact}
\sqrt n\,\mathscr{DS}_1
=
\overline{Y}_n\sqrt n(\widehat{\alpha}-\alpha)
+(1+\alpha)\sqrt n(\overline{Y}_n-\mu_Y)- \sqrt n(\overline{Z}_n-\mu_Z) .
\end{equation}
Using the Bahadur representation, and $\overline{Y}_n\to\mu_Y$ almost surely and $\overline{Z}_n\to\mu_Z$ almost surely, \eqref{eq:DS1compact} yields
\[
\sqrt n\,\mathscr{DS}_1
=
\frac1{\sqrt n}\sum_{i=1}^n \xi_i+o_{\mathbb P}(1),
\]
where
\[
\xi_i=
-\frac{\alpha^3}{(1+\alpha)^2}\left(\log X_i-\frac1{\alpha}\right)
+(1+\alpha)\left(\frac{\log X_i}{X_i}-\frac{\alpha}{(1+\alpha)^2}\right)-\bigg( \frac{1}{X_i} - \frac{\alpha}{1+\alpha} \bigg),
\]
are centred random variables. Now $\mathbb V(\log X_1)=1/\alpha^2$, $\mathbb V( \frac{\log X_1}{X_1})= \frac{2\alpha}{(2+\alpha)^3} - \frac{\alpha^2}{(1+\alpha)^4}$, and $\mathbb V(1/ X_1)= \frac{\alpha}{(2+\alpha) (1+\alpha)^2}$ . Further,
\begin{align*}
\mathrm{Cov}\bigg(\log X_1,\frac{\log X_1}{X_1}\bigg)
&=
\frac{\alpha-1}{(1+\alpha)^3}, \qquad     \mathrm{Cov}\bigg(\log X_1,\frac{1}{X_1}\bigg)
=
-\frac{1}{(1+\alpha)^2}, \text{ and} \\
\mathrm{Cov}\bigg(\frac{\log X_1}{X_1},\frac{1}{X_1}\bigg)
&=
\frac{\alpha}{(2+\alpha)^2} - \frac{\alpha^2}{(1+\alpha)^3}.
\end{align*}

Therefore, $\sigma^2=\mathbb V(\xi_1)=\frac{\alpha [2+(2 \alpha + \alpha^2 ) (5+8\alpha+4\alpha^2)]}{(2+\alpha)^3 (1+\alpha)^4 }$. Since the $\xi_i$ are iid.\ with mean $0$ and finite variance, the classical CLT and Slutsky's yield the claim.
\end{proof}

As a direct consequence, the following corollary provides a parameter free limit distribution.
\begin{cor}\label{cor:Tn}
    Let $\widehat{\alpha}$ be a consistent estimator of $\alpha$. Then under the standing assumptions we have
    \begin{equation*}
       (1+\widehat{\alpha})^2\sqrt{n\frac{(2+\widehat{\alpha})^3}{\widehat{\alpha} (2+(2\widehat{\alpha}+\widehat{\alpha}^2) (5+8\widehat{\alpha}+4\widehat{\alpha}^2))}}\, \mathscr{DS}_1\stackrel{\mathcal{D}}{\rightarrow} N(0,1)
    \end{equation*}
    as $n\rightarrow\infty$.
\end{cor}
So to perform the test, no bootstrap is needed.

\subsection{Asymptotic property of the test statistics \texorpdfstring{$\mathscr{DS}_2$}{DS2} under \texorpdfstring{$\mathscr{H}_0$ }{H0}}\label{limit2} 
\begin{thm}
    Let $X_1, \ldots, X_n$ be iid rvs following the Pareto distribution with shape parameter $\alpha$. Then the asymptotic distribution of $\sqrt n \mathscr{DS}_2$ is given by
    \[ \sqrt n \mathscr{DS}_2 \overset{d}{\rightarrow} \sup_{0<s<1} |\mathbb{G}(s) - Z_\alpha \cdot M(s)| , \]
    where $\mathbb{G}(s)$ is zero-mean Gaussian process with covariance function 
\[ \mathbb{C}\text{ov} (\mathbb{G}(s), \mathbb{G}(t)) = \mathbb{C}\text{ov} \bigg( \bigg( \frac{\alpha+1}{X} \cdot \frac{s-s^X}{-\log (s)} - s^X \bigg), \bigg( \frac{\alpha+1}{X} \cdot \frac{t-t^X}{-\log (t)} - t^X \bigg) \bigg),   \]
$Z_\alpha \sim N(0, \alpha^2)$ is the limiting distribution of $\sqrt{n} (\hat{\alpha} - \alpha)$, and 
\[M(s) = \mathbb{E} \bigg( \frac{\partial \phi}{\partial \alpha} (X,s;\alpha)   \bigg), \]
with \[\phi (x, s; {\alpha}) = \frac{\alpha+1}{x} \cdot  \frac{s-s^x}{-\log(s)}  .\]

\end{thm}

\begin{proof}
    Consider the statistic
    \begin{equation} \tag{*} \label{DS2}
     DS_2 = \frac{1}{n} \sum_{i=1}^n \phi (X_i, s; \hat{\alpha}) - \frac{1}{n} \sum_{i=1}^n s^{X_i},    
    \end{equation}
     and define  \[\mathscr{DS}_2 = \sqrt{n} \sup_{0<s<1} |DS_2| \]
Using a first order series expansion around $\alpha$, we have 
\[ \phi(X_i,s;\hat{\alpha}) =\phi(X_i,s;{\alpha}) + (\hat{\alpha}-\alpha) \frac{\partial \phi}{\partial \alpha} (X_i,s;\alpha) + o_p(1). \]

\noindent Substituting into Equation \eqref{DS2}, we get 
\begin{align*}
    DS_2 = \frac{1}{n} \sum_{i=1}^n \bigg( \phi(X_i,s;\alpha)-s^{X_i} \bigg) + (\hat{\alpha}-\alpha) \frac{1}{n} \sum_{i=1}^n \frac{\partial \phi}{\partial \alpha} (X_i,s;\alpha) + o_p(1),
\end{align*}

\noindent which implies,
 \begin{align*}
   \sqrt{n} DS_2 = \sqrt{n} \bigg( \frac{1}{n} \sum_{i=1}^n \bigg( \phi(X_i,s;\alpha)-s^{X_i} \bigg) \bigg)  + \sqrt{n} (\hat{\alpha}-\alpha) \frac{1}{n} \sum_{i=1}^n \frac{\partial \phi}{\partial \alpha} (X_i,s;\alpha) + o_P(1).
\end{align*}
Now define, 
\[ \mathbb{G}_n(s) = \sqrt{n} \bigg( \frac{1}{n} \sum_{i=1}^n N_s(X_i) - \mathbb{E} (N_s(X)) \bigg) ~  \text{with} ~  N_s(X) = \phi(X,s;\alpha)-s^X. \]   
Here $\mathbb{G}_n(s)$ is stochastic process indexed by $s$, and by central limit theorem it converges weakly to a zero-mean Gaussian process $\mathbb{G}
(s)$ ( see \citet{Vaart_1998}, theorem 19.28), with covariance function
\begin{align*}
   \mathbb{C}\text{ov} (\mathbb{G}(s), \mathbb{G}(t)) &= \mathbb{C}\text{ov} (N_s(X), N_t(X)), \\
   &= \mathbb{C}\text{ov} \bigg( \bigg( \frac{\alpha+1}{X} \cdot \frac{s-s^X}{-\log (s)} - s^X \bigg), \bigg( \frac{\alpha+1}{X} \cdot \frac{t-t^X}{-\log (t)} - t^X \bigg) \bigg). 
\end{align*}
From standard asymptotics, we know that
\[ \sqrt{n} (\hat{\alpha}-\alpha) \overset{d}{\rightarrow} N(0,\alpha^2).  \]
So, second term becomes $Z_\alpha \cdot M(s)$, where, 
\begin{align*}
    M(s) &= \frac{1}{n} \sum_{i=1}^n \frac{\partial \phi}{\partial \alpha} (X_i,s;\alpha), \\
    &= \mathbb{E} \bigg( \frac{\partial \phi}{\partial \alpha} (X,s;\alpha) \bigg), \\
    &= \frac{\alpha \, ( (\alpha+1) \, E_{\alpha+2} ( -\log(s)) -s  )}{(\alpha+1) \log(s)}.
\end{align*}
Where, \[E_n(z)=\int_1^{\infty }\frac{e^{-zt}}{t^n}dt .\]

\noindent Hence, 
 \[ \sqrt n \mathscr{DS}_2 \overset{d}{\rightarrow} \sup_{0<s<1} |\mathbb{G}(s) + Z_\alpha \cdot M(s)|  \]
\end{proof}

\subsection{Asymptotic property of the test statistics \texorpdfstring{$\mathscr{DS}_3$}{DS3} under \texorpdfstring{$\mathscr{H}_0$ }{H0} } \label{limit3} 
Due to the $L^2$- structure of the test statistic, a convenient setting the separable Hilbert space $\mathbb{H}$= $L^2 ([1,\inf),\mathscr{B},dt)$ of measurable functions $f: [1,\infty) \rightarrow \mathcal{R}$ satisfying $\int_1^\infty |f(t)|^2 \, dt < \infty$. Here, $\mathscr{B}$ stands for the Borel sigma-field on $[1,\infty)$. The scalar product and the norm in $\mathbb{H}$ will be denoted by
\[ \left< f,g \right>_\mathbb{H} = \int_1^\infty f(t) g(t) dt, \quad ||f||_\mathbb{H} =\left< f,f \right>_\mathbb{H}^{1/2}, \quad f,g \in \mathbb{H}, \]
respectively. Denote
\begin{align*}
Z_n(t) =& \frac{1}{\sqrt{n}} \sum_{i=1}^n \bigg( \frac{\hat{\alpha}+1}{X_i} \cdot \frac{\exp(-t) - \exp (-tX_i)}{t}  \bigg) - \frac{1}{\sqrt{n}} \sum_{i=1}^n \exp(-tX_i) ,
\end{align*}
so obviously $n \mathscr{DS}_3 = \int_1^\infty Z_n^2(t) \, dt$. After some calculation we have
\begin{align*}
Z_n(t)=& \frac{1}{\sqrt{n}} (\hat{\alpha}-\alpha) \sum_{i=1}^n \frac{1}{X_i} \bigg( \frac{\exp(-t) - \exp (-tX_i)}{t} \bigg) + \frac{\alpha+1}{\sqrt{n}} \sum_{i=1}^n \frac{1}{X_i} \bigg( \frac{\exp(-t) - \exp (-tX_i)}{t} \bigg)  \\ &- \frac{1}{\sqrt{n}} \sum_{i=1}^n \exp(-t X_i), 
\end{align*}
and in view of Eq.~\eqref{LRA}
\begin{align*}
\hat{Z}_n(t) =& \frac{1}{\sqrt{n}} \cdot \frac{\alpha^2}{n} \sum_{j=1}^n \bigg( \log X_j - \frac{1}{\alpha} \bigg) \cdot \sum_{i=1}^n \frac{1}{X_i} \bigg( \frac{\exp(-t) - \exp (-tX_i)}{t} \bigg) \\ &+ \frac{\alpha+1}{\sqrt{n}} \sum_{i=1}^n \frac{1}{X_i} \bigg( \frac{\exp(-t) - \exp (-tX_i)}{t} \bigg)  - \frac{1}{\sqrt{n}} \sum_{i=1}^n \exp(-t X_i), \\
=& \frac{\alpha^2}{\sqrt{n}} \psi_\alpha(t)  \sum_{j=1}^n \bigg( \log X_j - \frac{1}{\alpha} \bigg) + \frac{\alpha+1}{\sqrt{n}} \sum_{i=1}^n \frac{1}{X_i} \bigg( \frac{\exp(-t) - \exp (-tX_i)}{t} \bigg)  - \frac{1}{\sqrt{n}} \sum_{i=1}^n \exp(-t X_i),
\end{align*}
where, $\psi_\alpha(t) = \mathbb{E} \bigg( \frac{1}{X_i} \cdot \frac{\exp(-t) - \exp(-tX_i)}{t} \bigg) = \frac{\alpha \exp(-t) \, Ei(2+\alpha\, ,\, t) }{t}$.
\begin{thm}
    Under the standing assumptions, there exists a centered Gaussian process $Z$ in $\mathbb{H}$ with covariance kernel $C_\alpha(s,t) = \mathbb{E}\bigg(\widehat{Z}_1(s) \widehat{Z}_1(t)\bigg), \, s,t>1,$ such that 
    \[ \mathscr{DS}_3 \overset{d}{\rightarrow} ||Z||^2_\mathbb{H}, \quad \text{as} \quad n \rightarrow \infty\]
\end{thm}
\begin{proof}
The proof of the theorem is carried out in two stages:
\begin{itemize}
\item[1.] We first show that $||Z_n- \hat{Z}_n ||^2_\mathbb{H} = o_P(1)$. Note that $||\psi_\alpha||_\mathbb{H} <  \infty$, and  
\end{itemize}
\begin{align*}
    Z_n(t) - \hat{Z}_n(t) =& \frac{1}{\sqrt{n}} (\hat{\alpha}-\alpha) \sum_{i=1}^n \frac{1}{X_i} \bigg( \frac{\exp(-t) - \exp (-tX_i)}{t} \bigg) - \frac{\alpha^2}{\sqrt{n}} \psi_\alpha(t) \sum_{j=1}^n \bigg( \log X_j - \frac{1}{\alpha} \bigg),\\
    =& \sqrt{n} \,  (\hat{\alpha}-\alpha) \bigg( \frac{1}{n} \sum_{i=1}^n \frac{1}{X_i} \bigg( \frac{\exp(-t) - \exp (-tX_i)}{t} \bigg) - \psi_\alpha(t) \bigg),
\end{align*}
then,
\[ ||Z_n- \hat{Z}_n ||^2_\mathbb{H} = (\sqrt{n} (\hat{\alpha}-\alpha))^2 \int_{1}^\infty \bigg|\frac{1}{n} \sum_{i=1}^n \frac{1}{X_i} \bigg( \frac{\exp(-t)-\exp(-tX_i)}{t} \bigg) - \psi_\alpha(t) \bigg|^2 \, dt = o_P(1). \]
Above equation follows since $\sqrt{n} (\hat{\alpha}-\alpha)$ is $O_P(1)$ and $\int_{1}^\infty \bigg|\frac{1}{n} \sum_{i=1}^n \frac{1}{X_i} \bigg( \frac{\exp(-t)-\exp(-tX_i)}{t} \bigg) - \psi_\alpha(t) \bigg|^2 \, dt$ converges almost surely to 0 by the law of large numbers in Hilbert spaces because $\mathbb{E}\bigg(\frac{1}{X_i} \bigg( \frac{\exp(-t)-\exp(-tX_i)}{t} \bigg)\bigg)= \psi_\alpha(t)$ for $t>1$.   
\item[2.] Next, we show that $\widehat{Z}_n \overset{d}{\rightarrow} Z$ in $\mathbb{H}$. Since $||Z_n - \widehat{Z}_n||_\mathbb{H} = o_P(1)$ the limit behavior of $Z_n$ and $\widehat{Z}_n$ coincide. Furthermore, 
\[ \widehat{Z}_n(t) = \frac{1}{\sqrt{n}} \sum_{i=1}^n h(X_i, t),  \]
where
\[ h(x,t) = \alpha^2 \psi_\alpha(t) \bigg( \log x - \frac{1}{\alpha} \bigg) + \frac{\alpha+1}{tx} (\exp(-t)-\exp(-tx)) - \exp(-tx) \]
Here, $\mathbb{E} (h(x,t)) = 0$, i.e., we have a sum of iid centered random elements in $\mathbb{H}$. \\

\noindent The detailed calculations yield:
\begin{align*}
    \mathbb{E} (\widehat{Z}_1(t) \, \widehat{Z}_1(s)) =& \alpha^2 \psi_\alpha(t) \,  \psi_\alpha(s)- \alpha^2 \psi_\alpha(t) \bigg( \alpha \, G^{3,0}_{2,3} \left( s \middle| 
\begin{array}{c}
\alpha+1, \alpha+1 \\[6pt]
0, \alpha, \alpha
\end{array}
\right) - s^\alpha \Gamma (-\alpha, s) \bigg) \\ &+ \alpha^2 \psi_\alpha(t) \frac{\alpha+1}{s} \bigg( s^{\alpha+1} \, \Gamma (-1-\alpha,s) - \frac{\exp(-s)}{(\alpha+1)^2} - \alpha \,  G^{3,0}_{2,3} \left( s \middle| 
\begin{array}{c}
\alpha+2, \alpha+2 \\[6pt]
0, \alpha+1, \alpha+1
\end{array}
\right) \bigg)  \\ & + \alpha^2 \psi_\alpha(s) \,  \frac{\alpha+1}{t} \bigg( t^{\alpha+1} \, \Gamma (-1-\alpha,t) - \frac{\exp(-t)}{(\alpha+1)^2} - \alpha \,  G^{3,0}_{2,3} \left( t \middle| 
\begin{array}{c}
\alpha+2, \alpha+2 \\[6pt]
0, \alpha+1, \alpha+1
\end{array}
\right) \bigg) \\ &+
\frac{\alpha (\alpha+1)^2}{ts} \bigg( \frac{\exp (-s-t)}{2+\alpha} - \exp(-t) s^{\alpha+2} \Gamma (-\alpha-2, s) - \exp(-s) t^{\alpha+2} \Gamma (-\alpha-2, t) \\&+ (s+t)^{\alpha+2} \Gamma (-\alpha-2, s+t)  \bigg) + \alpha \, (s+t)^\alpha \Gamma(-\alpha, s+t) \\ &- \frac{\alpha+1}{t} \bigg( \alpha \exp(-t) s^{\alpha+1} \Gamma(-\alpha-1, s) - \alpha \, (s+t)^{\alpha+1} \Gamma(-\alpha-1, s+t) \bigg) \\ &- 
\alpha^2 \psi_\alpha(s) \bigg( \alpha \, G^{3,0}_{2,3} \left( t \middle| 
\begin{array}{c}
\alpha+1, \alpha+1 \\[6pt]
0, \alpha, \alpha
\end{array}
\right) - t^\alpha \Gamma (-\alpha, t) \bigg) \\ & - \frac{\alpha+1}{s} \bigg( \alpha \exp(-s) t^{\alpha+1} \Gamma(-\alpha-1, t) - \alpha \, (s+t)^{\alpha+1} \Gamma(-\alpha-1, s+t) \bigg).
\end{align*}
where, $G^{m,n}_{p,q} \left( z \middle| 
\begin{array}{c}
a_1,\ldots, a_p \\[6pt]
b_1,\ldots, b_q
\end{array}
\right) $ denotes the Meijer G-function, which can be expressed as,
\[G^{m,n}_{p,q} \left( z \middle| 
\begin{array}{c}
a_1,\ldots, a_p \\[6pt]
b_1,\ldots, b_q
\end{array}
\right) = \frac{r}{2 \pi \iota} \int \frac{\Gamma(1-a_1-rs)\ldots\Gamma(1-a_n-rs)\Gamma(b_1+rs)\ldots \Gamma(b_m+rs)}{\Gamma(a_{n+1}+rs)\ldots\Gamma(a_p+rs)\Gamma(1-b_{m+1}-rs)\ldots \Gamma(1-b_q-rs)} z^{-s} \, ds,  \]
where in the default case $r=1$ and $\Gamma(a,z)$ is the incomplete gamma function defined as 
\[\Gamma(a,z) = \int_z^\infty t^{a-1} e^{-t} \, dt.\]

\noindent An application of the CLT in Hilbert spaces shows $\widehat{Z}_n \overset{d}{\rightarrow} Z$ for $n \rightarrow \infty$.  \\

\noindent The proof of the statement follows by the continuous mapping theorem in Hilbert spaces.
\end{proof}

\section{Simulation Study} \label{MCstudy}
In this section, Monte Carlo simulations are used to compare the finite-sample performance of the proposed tests $\mathscr{DS}_{1},\mathscr{DS}_{2}$ and $\mathscr{DS}_{3}$ to the existing tests for the Pareto distribution. All simulations presented in the sequel were executed using the statistical computing environment \texttt{R}, as detailed in \citet{Rcore}. The codes utilized for the calculations are available from the authors upon request. We evaluate the performance of various gof tests considered in terms of the type I error rate by simulating samples of size $n= 20, ~ \text{and } 50 $, from the Pareto distribution $\mathcal{P}(\alpha)$, where we consider $\alpha \in \{1,2, 5\}$. On the other hand, to empirically investigate and compare the performance of the tests on power we simulate samples of size $n=20, ~ \text{and } 50 $ from a set of alternatives, that includes probability distributions having support $[1,\infty)$. These alternatives also incorporate distributions utilized in prior simulation studies for gof tests for Pareto distribution. The classes of alternative distributions considered are given in Table \ref{AD}. Note that several of the alternatives listed in this table have support $[0,\infty)$, and therefore we were required to shift these distributions by 1 unit to ensure that the simulated data have the same support as the Pareto distribution.

\begin{table}[htbp] 
\begin{center}
\caption{Summary of various choices of alternative distributions with support $[1,\infty)$.}
\label{AD}
\begin{tabular}{ l  c c }
\hline
Alternative & Probability Density Function & Notation \\
\hline
Gamma& $\frac{1}{\Gamma \theta} (x-1)^{\theta-1} \exp\left(-(x-1)\right)$ & $\Gamma(\theta)$ \\
Inverse-beta & $\frac{1+\theta}{x^2} \left(1- \frac{1}{x}\right)^ \theta$ & $InvB(\theta)$ \\
Tilted Pareto & $\frac{1+\theta}{(x+\theta)^2}$ & $Tilt(\theta)$ \\
Benini & $\frac{\exp\left(-\theta \log^2 x\right)}{x^2} (1+2\theta \log x)$ & $Ben(\theta)$ \\
Log-Gamma & $\frac{\log^\theta x}{x^2 \Gamma (1+\theta)}$ & $\log G(\theta)$\\ 
Log-Normal& $\frac{\exp \left(-\frac{1}{2} \left(\frac{\log (x-1)}{\theta}\right)^2\right)}{\theta (x-1) \sqrt{2 \pi}}$ & $LN(\theta)$ \\
Rayleigh & $\frac{x-1}{\theta^2} \exp\left(-\frac{(x-1)^2}{2\theta^2}\right)$ & $Ray(\theta)$\\
Weibull & $\theta (x-1)^{\theta-1} \exp\left(-(x-1)^\theta\right)$ & $W(\theta)$\\
L\'evy & $\sqrt{\frac{\theta}{2\pi}} \cdot \frac{\exp\left(-\frac{\theta}{2(x-1)}\right)}{(x-1)^{3/2}}$ & $L(\theta)$ \\
Burr($a,b,c$) & $cb \frac{\left(\frac{x-1}{a}\right)^{b-1}}{a \left(1+ {(\frac{x-1}{a})}^b\right) ^{c+1}}$ &$Burr(a,b,c)$ \\ 
Inverse-Gaussian & $\sqrt{\frac{\lambda}{2\pi (x-1)^3}} \exp\left(-\frac{\lambda(x-1-\mu)^2}{2\mu^2(x-1)}\right)$ & $IG(\mu,\lambda)$\\
Log-Weibull & $ \frac{(1+ \theta) \log^\theta x}{x^{2+\theta}}$ & $\log W(\theta)$\\
Fr\'echet & $\frac{a}{b} \left(\frac{x-1}{b}\right)^{-(a+1)} \exp \left( - \left(\frac{x-1}{b}\right)^{-a}\right)$ & $FR(a,b)$\\
Half-Normal & $\sqrt{\frac{2}{\pi^2 \theta^2}} \exp\left(-\frac{(x-1)^2}{2\theta^2}\right)$ & $HN(\theta)$\\
Chi-Square & $\frac{1}{2^{n/2} \Gamma (n/2)} (x-1)^{n/2-1} \exp\left(-(x-1)/2\right)$ & $\chi^2_n$ \\
Dhillon & $\frac{\theta+1}{x} \exp \left( -(\log(x) )^{\theta+1} \right) (\log(x))^ \theta$ & $DH(\theta)$ \\
Log-Logistic & $\frac{\gamma \left(\frac{x-1}{\theta}\right)^\gamma}{(x-1)[1+\left(\frac{x-1}{\theta}\right)^\gamma]^2}$ & $LLogis(\gamma,\theta)$\\ 
Linear failure rate & $\left(1+\theta(x-1)\right)\exp\left(-(x-1)-\theta(x-1)^2/2\right)$ & $LF(\theta)$\\ \hline
\end{tabular}
\end{center}
\end{table}

\noindent As competitors of the novel gof tests we considered the following existing tests for the Pareto distribution:
\begin{itemize}
\item The test proposed by \citet{obradovic2015asymptotic} given by:
\begin{equation*} 
	T_n=\int_{1}^{\infty} \left(M_n(t)-F_n(t)\right) dF_n(t),
\end{equation*}
where	
\begin{equation} \nonumber
	F_n(t) = \frac{1}{n} \sum_{i=1}^{n}\mathbb{I}(X_i\le t), 
\end{equation} 
and
\begin{equation}\nonumber
	\ {M}_n(t)=\binom{n}{2}^{-1} \sum_{i=1}^{n-1}\sum_{j=i+1}^{n}\mathbb{I}\left(  \max\left(  \dfrac{X_i}{X_j},  \dfrac{X_j}{X_i}\right)  \le t\right)  ,\quad t\ge 1.
\end{equation}

\item The test proposed by \citet{volkova2016goodness} given by
 \[I_n^{(k)} = \int_{1}^{\infty} (H_n(t)-F_n(t)) dF_n(t),\]
 where
 \[H_n(t) = \binom{n}{k}^{-1} \sum_{1 \leq i_1<\ldots<i_n \leq n} \mathbb{I} \left( X_{(k,\{i_1,\ldots,i_k\})} /X_{(k-1,\{i_1,\ldots,i_k\})} <t \right), \quad t \geq 1,\]
with $X_{(s,\{i_1,\ldots,i_k\})}$ to denote the $s-th$ order statistic of the subsample $X_{i_1},\ldots,X_{i_k}$.

\item the Kolmogorov-Smirnov $(KS_n)$ test, given by
\[KS_n = \max \bigg(\max_{1 \leq j \leq n} \bigg( \frac{j}{n} - F_{\hat{\alpha}_n} (X_{(j)}) \bigg) ,  \max_{1 \leq j \leq n} \bigg(F_{\hat{\alpha}_n} (X_{(j)})-\frac{j-1}{n}\bigg)     \bigg),\]

\item the Cramér-von Mises $(CM_n)$, given by
\[ CM_n = \frac{1}{12n} + \sum_{j=1}^n \bigg( F_{\hat{\alpha}_n} (X_{(j)}) - \frac{2j-1}{2n}\bigg)^2,  \]

\item the Anderson-Darling $(AD_n)$ test, given by
\[ AD_n = -n - \frac{1}{n} \sum_{j=1}^n (2j-1) \bigg( \log(F_{\hat{\alpha}_n} (X_{(j)})) + \log (1-F_{\hat{\alpha}_n} (X_{(n+1-j)})) \bigg), \]

\item  the tests proposed by \citet{zhang2002powerful}, which are defined by the following relations: 
 \[ZA_n = -\sum_{j=1}^{n} \left(\frac{\log \left( 1-X_{(j)}^{-\hat{\alpha}_n}\right)}{n-j+\frac{1}{2}} + \frac{\log \left( X_{(j)}^{-\hat{\alpha}_n}\right)}{j-\frac{1}{2}} \right), \] 
 \[ZB_n = \sum_{j=1}^{n} \left(\log\left(\frac{(1-X_{(j)}^{-\hat{\alpha}_n})^{-1}-1}{(n-0.5)/(j-0.75)-1}\right) \right)^2, \] 
and
\[ZC_n = 2\sum_{j=1}^{n} \frac{n(j-0.5)}{(n-j+0.5)^2}\log\left(\frac{j-0.5}{(1-X_{(j)}^{-\hat{\alpha}_n})n}\right)+2\sum_{j=1}^{n} \frac{n}{n-j+0.5}\log\left(\frac{n-j+0.5}{nX_{(j)}^{-\hat{\alpha}_n}}\right), \] 

\item the gof test $KL_{n,m}$ based on entropy which is defined (see \citet{ndwandwe2023newa}) by the following relation
\[
KL_{n,m}=-\frac{1}{n}\sum_{j=1}^{n}\log\left(\frac{n}{2m}\left(X_{(j+m)}-X_{(j-m)}\right)\right)-\log(\hat{\alpha}_n)+(\hat{\alpha}_n+1)\frac{1}{n}\sum_{j=1}^{n}\log(X_j).
\]
Note that this test includes a tuning parameter $m$ which was considered to be equal to 1 and 10 in the Monte Carlo study.

\item the Kullback-Leibler gof test ($DK_n$) which is defined (see \citet{ndwandwe2023newa}) by the following relation
\[
DK_n=\frac{1}{n}\sum_{j=1}^{n}\log\left(\frac{X_j^{\hat{\alpha}_n+1}\hat{f}_h(X_j)}{\hat{\alpha}_n}\right),
\]
where $\hat{f}_h(x)=\frac{1}{nh}\sum_{j=1}^{n}K\left(\frac{x-X_j}{h}\right)$ is the kernel density estimator with kernel function $K(\cdot)$ and bandwidth $h$. In the Monte Carlo study we choose the standard normal density function as a kernel, while $h=1.06 n^{-\frac{1}{5}}s$, with $s$ the sample standard deviation. Also, it is notable that this test is proved to be consistent against fixed alternatives.

\item The test proposed by \citet{meintanis2009unified}, which is based on Mellin transform and a special version of it is given by 
 \begin{multline*}
 	G_{n,a} = \frac{1}{n} \left( (\hat{\alpha}_n +1)^2 \sum_{j,k=1}^{n} I_a^{(0)} (X_j X_k) +\sum_{j,k=1}^{n} I_a^{(2)} (X_j X_k) +2(\hat{\alpha}_n +1) \sum_{j,k=1}^{n} I_a^{(1)} (X_j X_k) \right) \\
 	+ \hat{\alpha}_n \left(n \hat{\alpha}_n I_a^{(0)} (1) - 2(\hat{\alpha}_n +1) \sum_{j=1}^{n} I_a^{(0)} (X_j) -2 \sum_{j=1}^{n} I_a^{(1)} (X_j) \right),
 	\end{multline*}
 where 
\[
I_a^{(0)}(t) = (a+\log t)^{-1}, \text { }  I_a^{(1)}(t)= \frac{1-a-\log t}{(a+\log t)^{2}}
\]
and
\[
I_a^{(2)}(t)=	\frac{2-2a+a^2+2(a-1) \log t + \log^2 t}{(a+\log t)^{3}}.
\]
{Note that this test, which is proved to be consistent under fixed alternatives, includes a tuning parameter $a$ which was considered in the simulation study to be equal to 2}. 
\item The test proposed by \citet{meintanis2014class}, which is based on the empirical characteristic function and its simplified version, is defined by the following relation: 
	\begin{multline*}
	 M_{n,a} = \frac{1}{n} \sum_{j,k=1}^{n} \frac{2a}{(\hat{U_j}-\hat{U_k})^2+a^2} + 2n \left( 2 \tan^{-1} \left(\frac{1}{a}\right) - a \log \left( 1+ \frac{1}{a^2}\right)\right) \\ - 4\sum_{j=1}^{n} \left( \tan^{-1} \left(\frac{\hat{U_j}}{a}\right) + \tan^{-1} \left(\frac{1-\hat{U_j}}{a}\right)\right).
\end{multline*} 
where 
\[ \hat{U_j} = F_{\hat{\alpha}_n}(X_j)=1-X_{j}^{-\hat{\alpha}_n} ,\quad j=1,2,\ldots,n.\] 
{Note that this test, which is proved to be consistent under fixed alternatives, includes a tuning parameter $a$ which was considered to be equal to 2}.  
\item  The test proposed by \citet{allison2022distribution}, which is given by
\[I_{n,m} =\int_{1}^{\infty} \Delta_{n,m} (x) dF_n(x) ,\]
where 
\[ \Delta_{n,m} (x) = \frac{1}{n} \sum_{j=1}^{n} \mathbb{I} \bigg( X_j^{\frac{1}{m}} \leq x\bigg) -\frac{1}{n^m}  \sum_{j_1,j_2,\ldots,{j_m}=1}^{n} \mathbb{I} \bigg( \min(X_{j_1},X_{j_2},\ldots,{X_{j_m}}) \leq x\bigg).  \]
{Note that this test includes a tuning parameter $m$ which was considered to be equal to 2}.
\item We also considered four special cases of gof tests which belong to two different classes of tests proposed by \citet{ndwandwe2023new}. To be more specific we considered the tests 

\begin{multline*} S_{n,m,a}^{(1)} = \frac{1}{n} \sum_{j=1}^{n} \sum_{k=1}^{n} \left( \frac{2a}{a^2+( X_{(j)}^{1/m} -X_{(k)}^{1/m} )^2} -nv_{j,m} \frac{4a}{a^2+( X_{(j)} -X_{(k)}^{1/m} )^2}\right. \\  \left.
+ n^2 v_{j,m}v_{k,m} \frac{2a}{a^2+( X_{(j)} -X_{(k)})^2} \right), \end{multline*}
\begin{multline*}
	 S_{n,m,a}^{(2)} =  \frac{1}{n} \sqrt{\frac{\pi}{a}} \sum_{j=1}^{n} \sum_{k=1}^{n}  \bigg( \exp \left( \frac{-(X_{(j)}^{1/m} -X_{(k)}^{1/m} )^2}{4a} \right)-2nv_{j,m} \exp \left( \frac{-( X_{(j)} -X_{(k)}^{1/m} )^2}{4a} \right) \\ 
+ n^2 v_{j,m}v_{k,m} \exp \left( \frac{-( X_{(j)} -X_{(k)})^2}{4a} \right) \bigg),
\end{multline*}
\begin{multline*}
 T_{n,m,a}^{(1)} = \frac{1}{n} \sum_{j=1}^{n} \sum_{k=1}^{n}  \frac{2a}{a^2+( X_{(j)}^{1/m} -X_{(k)}^{1/m} )^2} -\sum_{j=1}^{n-m+1} \sum_{k=1}^{n} \binom{n}{m}^{-1} u_{j,m} \frac{4a}{a^2+( X_{(j)} -X_{(k)}^{1/m} )^2} \\
 + n \sum_{j=1}^{n-m+1} \sum_{k=1}^{n-m+1} \binom{n}{m}^{-2} u_{j,m} u_{k,m} \frac{2a}{a^2+( X_{(j)} -X_{(k)})^2}.  
\end{multline*}  
and
\begin{multline*}
	T_{n,m,a}^{(2)} = \sqrt{\frac{\pi}{a}} \frac{1}{n}  \sum_{j=1}^{n} \sum_{k=1}^{n}  \exp \left( \frac{-(X_{(j)}^{1/m} -X_{(k)}^{1/m} )^2}{4a} \right) \\ 
 -2 \sqrt{\frac{\pi}{a}} \sum_{j=1}^{n-m+1} \sum_{k=1}^{n} \binom{n}{m}^{-1} u_{j,m} \exp \left( \frac{-( X_{(j)} -X_{(k)}^{1/m} )^2}{4a} \right) \\ 
	+ n \sqrt{\frac{\pi}{a}} \sum_{j=1}^{n-m+1} \sum_{k=1}^{n-m+1}  \binom{n}{m}^{-2} u_{j,m}u_{k,m} \exp \left( \frac{-( X_{(j)} -X_{(k)})^2}{4a} \right).
\end{multline*} 
where
\[v_{j,m}:= \frac{1}{n^m} \bigg((n-j+1)^m -(n-j)^m\bigg) \text{ and }u_{j,m} = \binom{n-j}{m-1} . \]
Note that all tests, which are proved to be consistent under fixed alternatives, include two tuning parameters $m$ and $a$ which was considered to be equal to 3 and 2, respectively, in the simulation study.
\end{itemize}

\noindent Note that all the above tests reject the null hypothesis for large values of the test statistics, except for $I_{n,m}$, which rejects the null hypothesis for large absolute values of the test statistic. Additionally, the $T_n$ and $I_n^{(4)}$ tests are two-sided tests, meaning they reject the null hypothesis for test statistic values that are either significantly smaller or significantly larger than expected under the null hypothesis. Moreover, the empirical size and empirical power of the tests were computed by using the Monte Carlo bootstrap methodology. The algorithm used for finding the empirical size (power) is summarized as follows:
\begin{itemize}
		\item[1]  Draw a sample of size n, from null (alternative) distribution and calculate the test statistic. 
		\item[2] Find the maximum likelihood estimator $\hat{\alpha}$ of the parameter $\alpha$ using the data generated in Step 1. 
		\item[3] Generate a bootstrap sample of size $n$ by independently sampling from a $P(\hat{\alpha})$  distribution. Calculate the value of the test statistic using the bootstrap sample. 
        \item[4] Repeat Step 3, 1000 times and based on the test statistic value and 95\% quantile of bootstrap test statistic value, determine whether the null hypothesis is rejected or not. 
		\item[5] Repeat Steps 1-4, 1000 times and calculate the empirical size (power) as the proportion of rejections of the test.
	\end{itemize}

\begin{sidewaystable}[ht]
\caption{Empirical size and power roundest to the nearest integer of the proposed test and of some existing gof tests ($n=20$).} \label{power1}
\scalebox{0.7}{
\begin{tabular}{|l|cccccccccccccccccc|}
\hline
Distributions  & $T_n$ & $KS_n$ & $CM_n$ & $AD_n$ & $G_{n,2}$ & $M_{n,2}$ & $|I_{n,2}|$ & $S_{n,3,2}^{(1)}$ & $S_{n,3,2}^{(2)}$ & $T_{n,3,2}^{(1)}$ & $T_{n,3,2}^{(2)}$ & $N^{(1)}_{n,2,0.5}$ & $N^{(1)}_{n,3,0.5}$ & $N^{(2)}_{n,2,0.5}$ & $N^{(2)}_{n,3,0.5}$ & $\mathscr{DS}_1$ & $\mathscr{DS}_2$  & $\mathscr{DS}_3$  \\ \hline
$P(0.5)$ & 5& 5 & 4  & 5 & 5 & 5  & 4  & 3 & 4 & 4 & 4 & 5 & 4 & 4  & 4  & 5  & 4  & 5  \\
$P(1)$  & 5& 5  & 4  & 5 & 5 & 5  & 4 & 3  & 3 & 3 & 3  & 4 & 4 & 5& 4 & 5 & 5 & 5 \\
$P(2)$  & 5 & 5  & 4 & 4 & 5  & 5 & 4  & 3  & 3 & 2 & 2& 5 & 4 & 4 & 4 & 5  & 5 & 5 \\
$\Gamma(0.5)$  & 32 & 21 & 24& 46 & 14 & 19 & 30 & 6 & 6 & 17 & 16 & 20 & 27 & 16  & 26   & 12  & 11 & 9 \\
$\Gamma(1.2)$ & 54 & 46 & 56 & 49 & 61 & 61 & 50 & 62 & 61 & 54 & 51 & 45 & 53 & 59 & 58 & 62 & 64 & 55 \\
$\Gamma(2)$ & 100  & 96 & 100 & 99 & 100 & 100 & 99 & 100 & 100 & 99 & 100 & 95 & 99 & 99  & 100  & 100 & 100 & 100 \\
$Inv B(0.1)$ & 5 & 5 & 5 & 4 & 6 & 5 & 4 & 5 & 5 & 4 & 4 & 5 & 4 & 5 & 5 & 5 & 6  & 6 \\
$Inv B(2)$& 76  & 69 & 77 & 76  & 84 & 78  & 77  & 73 & 73  & 71  & 70  & 42 & 66  & 54 & 76 & 83 & 83 & 81  \\
$Tilt(1)$  & 12 & 12 & 13 & 11  & 13 & 13 & 11 & 13 & 13 & 11 & 11 & 8 & 9  & 10 & 13 & 11 & 13 & 13\\
$Tilt(2)$ & 23 & 21 & 25 & 21 & 25 & 26 & 22 & 24 & 23 & 22 & 22 & 12 & 15  & 17 & 22 & 22 & 26  & 27 \\
$Tilt(3)$ & 35 & 30  & 36  & 31 & 38 & 39 & 33 & 34 & 34 & 30 & 31 & 15 & 19  & 23 & 31& 29 & 36& 41 \\
$Ben(0.1) $& 6  & 7 & 6 & 5& 7  & 7 & 6  & 7  & 6 & 5 & 5 & 5 & 5 & 6& 6 & 6 & 8 & 6 \\
$Ben(0.3) $& 11  & 11& 13 & 11 & 13 & 13 & 10 & 14 & 14 & 11 & 10 & 9  & 12 & 12 & 13 & 13 & 13 & 10 \\
$Ben(0.7)$ & 20  & 18  & 22  & 18  & 24 & 23 & 19  & 26  & 25  & 20  & 18 & 18 & 21 & 21 & 23 & 24 & 24& 19 \\
$Log G(0.5)$ & 62 & 48  & 52 & 73 & 52 & 54  & 58 & 20 & 24 & 42 & 45 & 42 & 58 & 49 & 59  & 52 & 49  & 46 \\
$Log G(1.2)$ & 7 & 8 & 8  & 6  & 8& 8  & 7 & 8 & 7  & 7 & 6  & 8 & 8  & 7 & 8 & 8 & 9& 8 \\
$Log G(2)$ & 47 & 41 & 48 & 44  & 54 & 50 & 46& 44 & 45 & 41 & 41& 23  & 33 & 30 & 43  & 49  & 53  & 53 \\
$LN(0,1)$& 65 & 56& 66 & 64 & 72 & 70  & 65  & 68 & 67 & 63 & 61 & 46 & 66 & 55  & 72 & 72 & 71 & 61 \\
$LN(0,1.2)$ & 30 & 27 & 32  & 30 & 36  & 34  & 30  & 32 & 31 & 29 & 26& 22  & 31 & 25 & 35& 36 & 36 & 29 \\
$LN(0,1.5)$& 7 & 7 & 8  & 6  & 9  & 9  & 7  & 8  & 9  & 7 & 6 & 6  & 7 & 7  & 8& 8  & 9 & 8  \\
$Ray(1)$ & 100  & 100  & 100 & 100 & 100 & 100 & 100  & 100 & 100 & 100  & 100 & 100 & 100 & 100  & 100 & 100 & 100  & 100 \\
$W(0.5)$  & 42  & 31  & 35 & 58  & 28  & 29 & 41 & 14  & 11  & 28  & 22 & 27  & 43 & 29  & 39  & 32 & 25 & 19 \\
$W(1.2)$ & 61 & 50 & 62 & 58 & 67 & 66 & 57  & 68 & 67  & 61  & 60   & 57  & 62  & 68  & 64  & 69  & 68 & 61 \\
$W(1.75)$& 100& 95 & 100 & 99  & 100  & 100  & 98& 100 & 100 & 99  & 99 & 99 & 99 & 100 & 100  & 100  & 100  & 99 \\
$L(0,2)$& 22  & 20  & 23   & 23  & 31 & 21  & 24 & 25 & 22  & 23 & 21 & 19  & 25 & 17  & 27 & 35   & 32 & 26 \\
$Burr(1.5,0.5,0.5)$& 70  & 56 & 61  & 81  & 66 & 58 & 68  & 3 & 4 & 5   & 5  & 7  & 28 & 9  & 27 & 47 & 45  & 33  \\
$IG(1,1)$ & 78 & 75  & 80 & 81  & 81  & 81 & 80  & 81  & 80  & 77  & 75& 62  & 83  & 69  & 86  & 82  & 79 & 67  \\
$LogW (0.1)$ & 7  & 7  & 7  & 6  & 8 & 8  & 6 & 7  & 7 & 6  & 5 & 5 & 6  & 6 & 8  & 8 & 8 & 7 \\
$LogW (0.3)$ & 21 & 19  & 23  & 20 & 25 & 24 & 20 & 23  & 24  & 20  & 18 & 15 & 20& 20 & 24  & 25  & 26 & 21  \\
$LogW (1.75)$ & 100  & 100 & 100 & 100 & 100  & 100  & 100  & 100  & 100 & 100  & 100 & 100  & 100 & 100  & 100 & 100 & 100  & 100  \\
$FR(1,1)$ & 47  & 50 & 52  & 53 & 54 & 46 & 55 & 50  & 43  & 48 & 41 & 37   & 59  & 35  & 55 & 56  & 51 & 39  \\
$HN(0.8) $ & 57 & 49 & 57 & 53  & 61 & 62  & 50  & 65  & 63 & 56  & 53   & 62 & 56   & 63& 56 & 64 & 62 & 54 \\
$HN(1)$ & 63 & 52 & 65  & 59  & 68  & 68 & 57 & 71 & 69  & 62  & 60  & 66  & 61 & 72 & 62  & 70 & 69  & 62 \\
$\chi^2(4) $ & 100 & 100  & 100  & 100  & 100& 100  & 100  & 100  & 100  & 100  & 100  & 99  & 100 & 100 & 100  & 100  & 100   & 100 \\
$DH(0.2)$ & 13 & 12 & 14 & 12  & 14 & 13 & 13& 14& 14& 11  & 11  & 9 & 13  & 11  & 14 & 15 & 15 & 12  \\
$DH(0.4)$ & 33 & 28 & 34 & 30 & 40 & 39 & 31 & 36 & 36 & 31 & 30  & 21 & 32 & 30 & 39 & 40 & 41 & 34 \\
$DH(0.7)$ & 70 & 61 & 72 & 68 & 77 & 77 & 67 & 73 & 74 & 68 & 68 & 51 & 68 & 66 & 76 & 77 & 79 & 72 \\
$LLogis(2,1)$ & 87 & 80 & 87 & 86 & 88  & 87 & 86& 89 & 87 & 86 & 84 & 76 & 88 & 80 & 89 & 89 & 87 & 80 \\
$LF(0.2)$& 36 & 30 & 37 & 32 & 42 & 42 & 30 & 44 & 43 & 36 & 33 & 36 & 36 & 43 & 37 & 44 & 42 & 34 \\
$LF(0.5)$  & 44 & 37 & 46 & 40 & 50 & 49 & 39 & 52 & 51 & 44  & 42& 47 & 44 & 52 & 45 & 52 & 51 & 42\\
$LF(0.8)$& 50 & 41 & 52 & 45& 56 & 55 & 45 & 59& 57& 50 & 47 & 54 & 50  & 56 & 50 & 58 & 56 & 47 \\ \hline
\end{tabular}
}
\end{sidewaystable}

\begin{sidewaystable}[ht]
\caption{Empirical size and power roundest to the nearest integer of the proposed test and of some existing gof tests ($n=50$).} \label{power2}
\scalebox{0.7}{
\begin{tabular}{|l|cccccccccccccccccc|}
\hline
Distributions  & $T_n$ & $KS_n$ & $CM_n$ & $AD_n$ & $G_{n,2}$ & $M_{n,2}$ & $|I_{n,2}|$ & $S_{n,3,2}^{(1)}$ & $S_{n,3,2}^{(2)}$ & $T_{n,3,2}^{(1)}$ & $T_{n,3,2}^{(2)}$ & $N^{(1)}_{n,2,0.5}$ & $N^{(1)}_{n,3,0.5}$ & $N^{(2)}_{n,2,0.5}$ & $N^{(2)}_{n,3,0.5}$ & $\mathscr{DS}_1$ & $\mathscr{DS}_2$  & $\mathscr{DS}_3$  \\ \hline
$P(0.5)$& 6  & 6 & 6 & 8& 8 & 7 & 7 & 5 & 4 & 5 & 5 & 4 & 7 & 5 & 7 & 5 & 6 & 5\\
$P(1)$ & 5 & 4 & 4  & 4 & 8 & 7 & 7 & 5 & 5 & 5 & 5 & 5 & 7  & 6  & 6 & 5& 5 & 5 \\
$P(2)$ & 5 & 4 & 5 & 5 & 7 & 7  & 7 & 5 & 5  & 6  & 7& 6  & 6  & 7 & 6 & 6  & 5  & 5\\
$\Gamma(0.5)$ & 60 & 50 & 54 & 78 & 24 & 40  & 62 & 32& 27 & 45& 39 & 45 & 56  & 32& 52 & 21 & 19 & 15\\
$\Gamma(1.2)$& 94& 86 & 95 & 94 & 97 & 97 & 92 & 98 & 98 & 96 & 97 & 90 & 94 & 97 & 95 & 98& 98  & 98 \\
$\Gamma(2)$& 100 & 100 & 100  & 100 & 100 & 100 & 100 & 100 & 100 & 100 & 100 & 100 & 100 & 100 & 100 & 100 & 100 & 100 \\
$Inv B(0.1)$& 8 & 6& 8  & 7 & 9  & 8   & 7 & 8 & 8  & 7  & 7  & 6& 7& 6  & 8 & 8& 8 & 8 \\
$Inv B(2)$ & 99 & 98 & 99 & 100 & 100 & 99 & 100  & 98 & 97 & 97 & 97 & 84 & 97 & 92& 99 & 100 & 100 & 100 \\
$Tilt(1)$ & 22 & 20 & 24 & 21 & 25 & 25 & 23 & 22 & 23 & 20 & 20  & 12 & 16 & 17  & 22 & 24 & 25 & 25 \\
$Tilt(2)$ & 52  & 46 & 54& 49  & 53 & 56 & 50 & 46 & 49& 42 & 45 & 20 & 33 & 32  & 48 & 44 & 53 & 56 \\
$Tilt(3)$ & 72 & 64 & 73  & 70 & 72 & 77 & 71 & 62 & 66 & 59 & 63 & 29  & 49 & 45 & 65 & 62 & 70 & 77\\
$Ben(0.1)$ & 9 & 8 & 9 & 9 & 11 & 10 & 8 & 10 & 11 & 9 & 9  & 7  & 8  & 10  & 10 & 10 & 11 & 10 \\
$Ben(0.3)$ & 22  & 20 & 24 & 22& 26& 27 & 22 & 28 & 28 & 24& 25& 17  & 21  & 24 & 25  & 26 & 28 & 25 \\
$Ben(0.7)$ & 47 & 39 & 48 & 46 & 53 & 53 & 43 & 56  & 57 & 50 & 50 & 40 & 46 & 54 & 51& 54  & 55 & 52 \\
$Log G(0.5)$ & 93 & 85 & 89 & 97 & 88& 89 & 94 & 79 & 80 & 86 & 87 & 82 & 92 & 86 & 93 & 87 & 85& 81 \\
$Log G(1.2)$& 12 & 11 & 12 & 11 & 13 & 12 & 12 & 12 & 12 & 11 & 11 & 9 & 11 & 10 & 12 & 13 & 13 & 11 \\
$Log G(2)$ & 91& 84 & 89 & 91 & 95 & 92 & 92 & 79 & 79 & 77 & 77 & 46 & 75 & 61& 87  & 93 & 94  & 94 \\
$L N(0,1)$ & 99& 95 & 98 & 99  & 99 & 99 & 99 & 98 & 99& 98 & 98  & 90 & 98 & 95 & 99 & 99 & 99 & 98\\
$L N(0,1.2)$ & 71 & 58 & 70& 73 & 78 & 75 & 73 & 69 &73 & 65 & 66 & 42 & 70 & 55 & 76  & 76 & 75 & 70 \\
$L N(0,1.5)$ & 11  & 11  & 12 &11 & 15 & 14  & 11 & 12& 13 & 10 & 11 & 8& 10& 9  & 13 & 14  & 14 & 14 \\
$Ray(1)$ & 100 & 100 & 100 & 100  & 100  & 100 & 100 & 100 & 100 & 100 & 100 & 100 & 100 & 100 & 100  & 100 & 100 & 100 \\
$W(0.5)$ & 81 & 72 & 76 & 92 & 61 & 63 & 81 & 66 & 57 & 75 & 66  & 68 & 83 & 67 & 77 & 66 & 56 & 41 \\
$W(1.2)$ & 95 & 89& 96 & 96 & 97 & 97& 94 & 98 & 98 & 97 & 97 & 94 & 95 & 98  & 96 & 98 & 98 & 98\\
$W(1.75)$ & 100 & 100 & 100 & 100& 100 & 100 & 100 & 100 & 100& 100 & 100 & 100 & 100 & 100 & 100 & 100 & 100& 100 \\
$L(0,2)$& 55 & 54 & 55 & 68 & 70  & 44  & 65 & 54  & 47 & 53  & 46 & 39& 64 & 35 & 62 & 81 & 74 & 56 \\
$Burr(1.5,0.5,0.5)$ & 97 & 93 & 95 & 99 & 95 & 91 & 97  & 19 & 19 & 23 & 24 & 21 & 58 & 25 & 59 & 81 & 78 & 64 \\
$I G(1,1)$ & 100  & 100 & 100 & 100 & 100 & 100 & 100 & 100 & 100  & 100& 100 & 97  & 100 & 99 & 100 & 100 & 100& 99 \\
$LogW (0.1)$& 11 & 9 & 10 & 10 & 12 & 11  & 11 & 12 & 12 & 10 & 10& 8 & 10 & 10  & 11& 12 & 12 & 11 \\
$LogW (0.3)$ & 51& 41& 49 & 48& 55 & 54 & 50 & 50 & 52 & 46 & 46 & 30 & 45 & 40  & 54  & 55 & 55 & 51 \\
$LogW (1.75)$ & 100 & 100 & 100  & 100  & 100 & 100 & 100 & 100 & 100 & 100 & 100 & 100 & 100 & 100 & 100 & 100 & 100 & 100  \\
$FR(1,1)$& 92 & 94 & 94 & 97 & 89 & 81 & 96 & 90 & 84 & 89 & 83 & 79 & 97 & 74 & 92 & 93 & 88 & 75 \\
$HN(0.8) $ & 95 & 90 & 96& 95 & 98 & 97 & 91 & 98  & 98 & 96 & 96 & 97 & 94 & 99 & 93 & 99 & 99& 98 \\
$HN(1)$& 97 & 93 & 98 & 97 & 99 & 99 & 95 & 99 & 99 & 99 & 99& 98& 96 & 99 & 97 & 99 & 99 & 99  \\
$\chi^2(4)$ & 100 & 100 & 100  & 100  & 100 & 100 & 100 & 100  & 100  & 100  & 100 & 100  & 100 & 100  & 100  & 100  & 100& 100 \\
$DH(0.2)$ & 26 & 22 & 26 & 25 & 29 & 27 & 27 & 26 & 27 & 23 & 23 & 15 & 24 & 22 & 28 & 28 & 28 & 26  \\
$DH(0.4)$ & 72 & 62 & 71 & 71 & 78 & 77 & 73& 72 & 75 & 68 & 69 & 48 & 67 & 64 & 76 & 78 & 78& 75 \\
$DH(0.7)$ & 99 & 95 & 98 & 98 & 100 & 99 & 98 & 98 & 99 & 98 & 99 & 91 & 97& 97  & 99 & 99 & 100 & 99 \\
$LLogis(2,1)$ & 100  & 99   & 100    & 100 & 100 & 100  & 100 & 100 & 100 & 100 & 100  & 99 & 100  & 99 & 100 & 100 & 100  & 98 \\
$LF(0.2)$ & 74 & 65 & 78 & 75 & 84 & 83 & 67& 85& 85 & 80 & 80  & 79 & 74 & 88& 76 & 85 & 87 & 85  \\
$LF(0.5)$ & 84 & 75 & 87 & 86 & 91 & 90 & 79& 92 & 92 & 88 & 88 & 89 & 84 & 93 & 84 & 92 & 92 & 91 \\
$LF(0.8)$& 89 & 82 & 91 & 89 & 94 & 93 & 85 & 94 & 94 & 92 & 92 & 93 & 89 & 96 & 89 & 95 & 95 & 94 \\ \hline
\end{tabular}
}
\end{sidewaystable}

%-----------------------------------------
\section{Data Analysis} \label{data}
In this section, we demonstrate the practical application of the proposed test procedures using real-world data sets.\\

\noindent \textbf{Illustration 1:} 
We analyze the earnings data from the 2022 inaugural season of LIV Golf, presented in \citet{Ngatchou}, and detailed in Table~\ref{DGnew}. According to their analysis, the earnings exhibit a heavily right-skewed distribution, with a few exceptionally large values. Due to this characteristic, the Pareto distribution emerges as a natural and appropriate model for this dataset. Following the approach in \citet{Ngatchou}, our investigation concentrates on earnings that exceed the threshold of 3,500,000. The earnings above this level are selected and subsequently normalized by dividing each amount by 3,500,000.\\

\noindent Let $Y_1, Y_2, \ldots, Y_n$ denote the resulting transformed data, where each value is given by $Y_i = \left(\frac{X_i}{3{,}500{,}000}\right)^{\hat{\alpha}}$, for $i = 1, \ldots, n$, with the estimated shape parameter $\hat{\alpha}_n = 1.428$. Using these transformed values, we compute the proposed test statistic, denoted as \textit{TS}, and compare its performance with several alternative goodness-of-fit tests outlined in Section~\ref{MCstudy}.

\noindent To evaluate the significance of the observed test statistics, we employ a bootstrap method using resamples drawn from the fitted Pareto distribution $\mathcal{P}(\hat{\alpha})$. Table~\ref{SG} summarizes the test statistics along with their associated bootstrap $p$-values. The results indicate that all the $p$-values exceed the conventional 5\% significance level. Consequently, we do not reject the hypothesis that the earnings exceeding 3,500,000 in the 2022 LIV Golf season are consistent with a Pareto distribution. This outcome supports the findings in \citet{Ngatchou}, where no significant departure from the Pareto assumption was detected.

\begin{table}[H]
\begin{center}
\caption{LIV golf earnings data set (see \citet{Ngatchou} and references therein)}
\label{DGnew}
\scalebox{0.85}{
\begin{tabular}{ c c c c c c c }
\hline 
33,509,017 &16,700,083& 13,726,499& 13,254,785 &11,040,714 &9,297,500 &8,068,000 \\
8,033,500 & 7,982,785 &7,345,000 &7,268,167 &5,580,314 &5,518,500 &4,992,618\\
4,849,000  &4,596,000 & 4,454,500 &4,408,000 &4,230,964 &4,195,367 &3,726,583\\
3,558,666 &3,424,333 &3,411,314 &3,399,100& 3,248,000 &3,085,000& 2,948,500 \\
2,944,950& 2,676,500& 2,675,833& 2,569,917 &2,533,833 & 2,421,714& 2,379,700\\
2,188,285& 2,016,350& 1,899,000& 1,861,667& 1,573,000 &1,566,999 &1,471,200 \\
1,377,000 &1,366,000 &1,321,950 &1,274,285& 1,185,000 &1,097,000&1,050,000\\
1,043,000& 1,019,000 &958,750 &958,750 &840,000 &800,000 &800,000\\
683,000 &651,000 &540,000 &530,000 &437,000 &360,000 &324,000 \\
268,000 &250,000 &232,000 &208,750 &190,000 &160,000& 156,000\\
150,000&140,000 &140,000 &133,750& 120,000&&\\
\hline
\end{tabular}
}
\end{center}
\end{table}

\begin{table}[H] 
\begin{center}
\caption{Summary of results for the 2022 season's earnings of LIV golfers exceeding 3,500,000.}
\label{SG}
\scalebox{0.85}{
\begin{tabular}{ l c c c c c }
\hline
Test & TS Value & $p$-value & Test & TS Value & $p$-value \\ \hline
$T_n$ & 0.481  & 0.582 & $T_{n,3,2}^{(1)}$  & 0.001    & 0.615   \\
$KS_n$ & 0.148    & 0.457   & $T_{n,3,2}^{(2)}$   & 0.001    & 0.501   \\
$CM_n$ & 0.068    & 0.556   & $N^{(1)}_{n,2,0.5}$ & 0.380    & 0.670   \\
$AD_n$  & 0.391    & 0.661  & $N^{(1)}_{n,3,0.5}$  & 0.203   & 0.606   \\
$G_{n,2}$  & 0.019    & 0.302   & $N^{(2)}_{n,2,0.5}$  & 0.162    & 0.475   \\
$M_{n,2}$ & 0.011    & 0.327   & $N^{(2)}_{n,3,0.5}$& 0.137    & 0.360   \\
$|I_{n,2}|$  & 0.015    & 0.617   & $\mathscr{DS}_1$ & 0.078    & 0.383   \\
$S_{n,3,2}^{(1)}$  & 0.002  & 0.419   & $\mathscr{DS}_2$  & 0.003    & 0.296   \\
$S_{n,3,2}^{(2)}$  & 0.002    & 0.332   & $\mathscr{DS}_3$  & 0.068  & 0.292  \\    
\hline
\end{tabular}
}
\end{center}
\end{table}

%-----------------------------------------------
%-----------------------------------------------
\vspace{0.5cm}
\noindent \textbf{Illustration 2:} The second example examines the Airplane dataset, previously analyzed in \citet{linhart1986model}, which consists of 30 recorded failure times of an aircraft’s air conditioning system (see Table \ref{DA}). 

\begin{table}[H] 
\begin{center}
\caption{Airplane Data set}
\label{DA}
\begin{tabular}{ c c c c c c c c c c }
\hline
23 &261 &87 &7 &120 &14& 62& 47& 225& 71 \\
246& 21 &42& 20 &5& 12& 120& 11& 3 &14\\
71& 11& 14& 11& 16& 90& 1 &16& 52& 95\\
\hline
\end{tabular}
\end{center}
\end{table}
\noindent Following a procedure similar to that described in Illustration 1, we apply a transformation to the data and estimate the shape parameter as $\hat{\alpha}_n = 0.30$. The test statistic (TS) values are then evaluated, and bootstrap $p$-values are obtained by simulating samples from a Pareto distribution with parameter $\alpha = 0.30$. Since almost all resulting $p$-values fall below the 5\% level of significance, this leads to the rejection of the null hypothesis in every test. Consequently, the analysis suggests that the Pareto model does not provide an adequate fit for the failure times in the Airplane dataset.

\begin{table}[H]
\begin{center}
\caption{Summary of results for Airplane data set}
\label{SA}
\scalebox{0.8}{
\begin{tabular}{c c r c c r } 
\hline
Test  & TS Value & $p$-value & Test   & TS Value & $p$-value \\ \hline
$T_n$ & 8.895 & 0.000 & $T_{n,3,2}^{(1)}$& 0.448 & 0.000 \\
$KS_n$ & 0.377 & 0.000 & $T_{n,3,2}^{(2)}$  & 0.630 & 0.000 \\
$CM_n$ & 1.028 & 0.000 & $N^{(1)}_{n,2,0.5}$ & 3.802 & 0.010 \\
$AD_n$ & 4.771 & 0.000 & $N^{(1)}_{n,3,0.5}$ & 3.323 & 0.001 \\
$G_{n,2}$ & 0.179 & 0.001 & $N^{(2)}_{n,2,0.5}$ & 3.200 & 0.002 \\
$M_{n,2}$ & 0.168 & 0.000 & $N^{(2)}_{n,3,0.5}$ & 2.898 & 0.000 \\
$|I_{n,2}|$ & 0.189 & 0.000 & $\mathscr{DS}_1$ & 0.079 & 0.074 \\
$S_{n,3,2}^{(1)}$ & 0.454 & 0.000 & $\mathscr{DS}_2$ & 0.011 & 0.005 \\
$S_{n,3,2}^{(2)}$ & 0.638 & 0.000 & $\mathscr{DS}_3$  & 0.266 & 0.000\\
\hline 
\end{tabular}
}
\end{center}
\end{table}

\section{Conclusion and Open problems} \label{conclusions}
In this article, we proposed three new gof tests for the Pareto distribution, developed from a novel characterization using Stein’s method and the Laplace transform. The asymptotic properties of the test statistics under the null hypothesis were rigorously derived. To assess their practical performance, we conducted a comprehensive simulation study, comparing the proposed tests with several existing methods. The simulation results indicate that our tests often outperform the competitors in terms of power, suggesting their promise as robust tools for model validation in heavy-tailed settings.\\

\noindent While the current work focuses on scenarios where the scale parameter $\beta$ is known or standardized to 1, a natural extension would be to modify the proposed tests to accommodate the case where $\beta$ is unknown and must be estimated from the data. Addressing this will improve the applicability of the tests in more general situations.

\noindent Additionally, the methodology has been developed for complete data. Extending it to settings involving censored or truncated observations which is common in survival analysis and reliability theory, would be a significant step forward. This direction is particularly relevant given the increasing demand for gof procedures in contexts with incomplete data.

\noindent Moreover, the concept used in this paper can be extended beyond the Laplace transform. For example, one could investigate the use of other integral transforms, such as the Fourier transform or the Mellin transform, in conjunction with Stein’s method to develop new gof tests. This would help assess whether the integration of Stein-type forms into different functional domains enhances sensitivity to deviations from the hypothesized model.

\noindent Finally, applying this characterization-substitution framework to other distributions with known Stein identities (e.g., exponential, gamma, log-normal) may further broaden its applicability, offering a unified perspective for constructing gof tests across distribution families.

%-----------------------------------
\section{Acknowledgements} 
Sakshi Khandelwal would like to express sincere gratitude to the Department of Science and Technology (DST) for the generous support provided through the Inspire Fellowship  (Inspire code: IF220245).
%-----------------------------------
\bibliographystyle{apalike} 
\bibliography{bibitex}

@article{betsch2021fixed,
  title={Fixed point characterizations of continuous univariate probability distributions and their applications},
  author={Betsch, S. and Ebner, B.},
  journal={Annals of the Institute of Statistical Mathematics},
  volume={73},
  pages={31--59},
  year={2021},
  publisher={Springer}
}

@book{Vaart_1998,
title={Asymptotic Statistics}, 
place={Cambridge},
series={Cambridge Series in Statistical and Probabilistic Mathematics},  publisher={Cambridge University Press},
author={Van der Vaart, A. W. }, 
year={1998},
pages={265–290}}

@Manual{Rcore,
  title = {{R: A Language and Environment for Statistical Computing}},
  author = {R Core Team},
 organization = {R Foundation for Statistical Computing},
 address = {Vienna, Austria},
 year = {2021},
 url = {https://www.R-project.org/},
}

@article{zhang2002powerful,
  title={{Powerful goodness-of-fit tests based on the likelihood ratio}},
  author={Zhang, J.},
  journal={Journal of the Royal Statistical Society Series B: Statistical Methodology},
  volume={64},
  number={2},
  pages={281--294},
  year={2002},
  publisher={Oxford University Press}
}

@article{meintanis2009unified,
  title={{A unified approach of testing for discrete and continuous Pareto laws}},
  author={Meintanis, S.G.},
  journal={Statistical Papers},
  volume={50},
  pages={569--580},
  year={2009},
  publisher={Springer}
}

@article{meintanis2014class,
  title={{A class of goodness-of-fit tests based on transformation}},
  author={Meintanis, S.G. and Jim\'enez-Gamero, M. D. and Alba-Fern{\'a}ndez, V.},
  journal={Communications in Statistics-Theory and Methods},
  volume={43},
  number={8},
  pages={1708--1735},
  year={2014},
  publisher={Taylor \& Francis}
}

@article{allison2022distribution,
  title={{Distribution-free goodness-of-fit tests for the Pareto distribution based on a characterization}},
  author={Allison, J. and Milo{\v{s}}evi{\'c}, B. and Obradovi{\'c}, M. and Smuts, M.},
  journal={Computational Statistics},
  volume={37},
  number={1},
  pages={403--418},
  year={2022},
  publisher={Springer}
}

@article{ndwandwe2023new,
  title={On a new class of tests for the Pareto distribution using Fourier methods},
  author={Ndwandwe, L.M. and Allison, J.S. and Smuts, M. and Visagie, I.J.H.},
  journal={Stat},
  volume={12},
  number={1},
  pages={e566},
  year={2023},
  publisher={Wiley Online Library}
}

@article{volkova2016goodness,
  title={{Goodness-of-fit tests for the Pareto distribution based on its characterization}},
  author={Volkova, K.},
  journal={Statistical Methods \& Applications},
  volume={25},
  number={3},
  pages={351--373},
  year={2016},
  publisher={Springer}
}

@article{ndwandwe2023newa,
  title={Testing for the Pareto type I distribution: A comparative study},
  author={Ndwandwe, L.  and Allison, J.S. and Santana, L. and Visagie, I.J.H.},
  journal={Metron},
  volume={81},
  pages={215-256},
  year={2023},
}

@article{obradovic2015asymptotic,
  title={{On asymptotic efficiency of goodness of fit tests for Pareto distribution based on characterizations}},
  author={Obradovi{\'c}, M.},
  journal={Filomat},
  volume={29},
  number={10},
  pages={2311--2324},
  year={2015},
  publisher={JSTOR}
}

@book{pareto1896cours,
  title={{Cours d'{\'e}conomie politique: profess{\'e} {\`a} l'Universit{\'e} de Lausanne}},
  author={Pareto, V.},
  volume={1},
  year={1896},
  publisher={F. Rouge}
}

@article{simon1955class,
  title={{On a class of skew distribution functions}},
  author={Simon, H. A.},
  journal={Biometrika},
  volume={42},
  number={3/4},
  pages={425--440},
  year={1955},
  publisher={JSTOR}
}

@article{rydberg2000realistic,
  title={{Realistic statistical modelling of financial data}},
  author={Rydberg, T. H.},
  journal={International Statistical Review},
  volume={68},
  number={3},
  pages={233--258},
  year={2000},
  publisher={Wiley Online Library}
}

@article{malik1970characterization,
  title={{A characterization of the Pareto distribution}},
  author={Malik, H. J.},
  journal={Scandinavian Actuarial Journal},
  volume={1970},
  number={3-4},
  pages={115--117},
  year={1970},
  publisher={Taylor \& Francis}
}

@article{ahsanullah1974characterization,
  title={{A characterization of the power function distribution}},
  author={Ahsanullah, M. and Kabir, A.B.M.L.},
  journal={The Canadian Journal of Statistics/La Revue Canadienne de Statistique},
  pages={95--98},
  year={1974},
  publisher={JSTOR}
}

@article{nofal2017new,
  title={{New characterizations of the Pareto distribution}},
  author={Nofal, Z.M. and El Gebaly, Y. M.},
  journal={Pakistan Journal of Statistics and Operation Research},
  pages={63--74},
  year={2017},
  publisher={College of Statistical and Actuarial Sciences}
}

@article{Ngatchou,
author = {Ngatchou–Wandji,J. and Nombebe, T. and Santana, L. and Allison, J.S.},
title = {{On classes of consistent tests for the Type I Pareto distribution based on a characterization involving order statistics}},
journal = {Statistics},
volume = {58},
number = {3},
pages = {521--551},
year = {2024},
publisher = {Taylor \& Francis},
doi = {10.1080/02331888.2024.2347342},
}

@book{linhart1986model,
  title={Model selection.},
  author={Linhart, Heinz and Zucchini, Walter},
  year={1986},
  publisher={John Wiley \& Sons}
}

@article{ebner2023goodness,
  title={Goodness-of-fit tests for the Weibull distribution based on the Laplace transform and Stein’s method},
  author={Ebner, Bruno and Fischer, Adrian and Henze, Norbert and Mayer, Celeste},
  journal={Annals of the Institute of Statistical Mathematics},
  volume={75},
  number={6},
  pages={1011--1038},
  year={2023},
  publisher={Springer}
}

@book{embrechts1997modelling,
  title     = {Modelling Extremal Events: for Insurance and Finance},
  author    = {Embrechts, Paul and Kl{\"u}ppelberg, Claudia and Mikosch, Thomas},
  year      = {1997},
  publisher = {Springer},
  address   = {Berlin, Heidelberg}
}

@article{katz2002statistics,
  title={Statistics of extremes in hydrology},
  author={Katz, Richard W and Parlange, Marc B and Naveau, Philippe},
  journal={Advances in water resources},
  volume={25},
  number={8-12},
  pages={1287--1304},
  year={2002},
  publisher={Elsevier}
}

@article{choulakian2001goodness,
  title={Goodness-of-fit tests for the generalized Pareto distribution},
  author={Choulakian, Vartan and Stephens, Michael A},
  journal={Technometrics},
  volume={43},
  number={4},
  pages={478--484},
  year={2001},
  publisher={Taylor \& Francis}
}

@book{castillo2012extreme,
  title     = {Extreme Value and Related Models with Applications in Engineering and Science},
  author    = {Castillo, Enrique and Hadi, Ali S. and Balakrishnan, N. and Sarabia, Jos{\'e} M.},
  year      = {2012},
  publisher = {John Wiley \& Sons},
  address   = {Hoboken, NJ},
  isbn      = {978-0-470-17026-6}
}

@inproceedings{stein1972bound,
  author    = {Stein, Charles},
  title     = {A Bound for the Error in the Normal Approximation to the Distribution of a Sum of Dependent Random Variables},
  booktitle = {Proceedings of the Sixth Berkeley Symposium on Mathematical Statistics and Probability},
  volume    = {2},
  editor    = {Le Cam, Lucien M. and Neyman, Jerzy and Scott, Elizabeth L.},
  pages     = {583--602},
  publisher = {University of California Press},
  address   = {Berkeley},
  year      = {1972},
  doi       = {10.1525/9780520423671-036}
}

@article{kumari2023goodness,
  title={Goodness-of-fit test for one-sided l{\'e}vy distribution based on stein’s characterization},
  author={Kumari, Aditi and Sudheesh, KK and Bhati, Deepesh},
  journal={Journal of the Indian Society for Probability and Statistics},
  volume={24},
  number={2},
  pages={377--391},
  year={2023},
  publisher={Springer}
}

@article{bhati2025new,
  title={New goodness of fit tests for the Pareto distribution using Stein’s characterization for uncensored and random right censored data},
  author={Bhati, Deepesh and Batsidis, Apostolos and Khandelwal, Sakshi},
  journal={Communications in Statistics-Theory and Methods},
  pages={1--25},
  year={2025},
  publisher={Taylor \& Francis}
}

@book{widder1941laplace,
  title     = {The Laplace Transform},
  author    = {Widder, David Vernon},
  year      = {1941},
  publisher = {Princeton University Press},
  address   = {Princeton, NJ}
}

@article{baringhaus1991class,
  title={A class of consistent tests for exponentiality based on the empirical Laplace transform},
  author={Baringhaus, Ludwig and Henze, Norbert},
  journal={Annals of the Institute of Statistical Mathematics},
  volume={43},
  pages={551--564},
  year={1991},
  publisher={Springer}
}

@article{henze1992new,
  title={A new flexible class of omnibus tests for exponentiality},
  author={Henze, Norbert},
  journal={Communications in Statistics-Theory and Methods},
  volume={22},
  number={1},
  pages={115--133},
  year={1992},
  publisher={Taylor \& Francis}
}

@article{jimenez2009goodness,
  title={Goodness-of-fit tests based on empirical characteristic functions},
  author={Jim{\'e}nez-Gamero, Mar{\'\i}a-Dolores and Alba-Fern{\'a}ndez, V and Mu{\~n}oz-Garc{\'\i}a, Joaqu{\'\i}n and Chalco-Cano, Yurilev},
  journal={Computational Statistics \& Data Analysis},
  volume={53},
  number={12},
  pages={3957--3971},
  year={2009},
  publisher={Elsevier}
}

@book{nikitin1995asymptotic,
  title     = {Asymptotic Efficiency of Nonparametric Tests},
  author    = {Nikitin, Ya. Yu.},
  year      = {1995},
  publisher = {Cambridge University Press},
  address   = {Cambridge},
  doi       = {10.1017/CBO9780511530081}
}

\end{document}